\documentclass[a4paper]{article}
\usepackage{fancyhdr}
\usepackage{latexsym}
\usepackage{amssymb}
\usepackage{mathtools}

\usepackage{abstract}

\usepackage[subnum]{cases}

\usepackage{amstext}
\usepackage{amsfonts}
\usepackage{mathrsfs}
\usepackage{multirow}
\usepackage{multicol}
\usepackage{stackrel}

\usepackage[square]{natbib}
\setcitestyle{numbers}

\usepackage[latin1]{inputenc}

\usepackage{csquotes}

\usepackage{float}
\usepackage{verbatim}
\usepackage{longtable}
\usepackage{tabularx}       
\usepackage{multicol}
\usepackage{color}
\usepackage{wrapfig}
\usepackage{floatflt}
\usepackage{afterpage}

\usepackage{textcomp}
\usepackage[gen]{eurosym}

\usepackage{amsmath}
\usepackage{mathabx}
\usepackage{amssymb}
\usepackage{amstext}
\usepackage{mathrsfs}
\usepackage{amsfonts}
\usepackage{bbm}

\usepackage{graphicx}
\usepackage{epsfig}
\usepackage{subfigure}

\usepackage{nomencl}
\usepackage{epigraph}
\makenomenclature

\usepackage{url}
\usepackage{appendix}

\usepackage{amsthm}
\usepackage{yfonts}

 \def\be{\begin{eqnarray}}
\def\ee{\end{eqnarray}}
\def\b*{\begin{eqnarray*}}
\def\e*{\end{eqnarray*}}

\newcommand{\R}{\mathbb{R}}

\newcommand{\N}{\mathbb{N}}
\newcommand{\PP}{\mathbb{P}}

\newcommand{\dd}{\mathrm{d}}
\newcommand{\cadlag}{c\`adl\`ag }

\newcommand{\AY}{Az\'{e}ma-Yor }
\newcommand{\iAY}{iterated Az\'{e}ma-Yor }

\newcommand\restr[2]{{
  \left.\kern-\nulldelimiterspace 
  #1 
  \vphantom{\big|} 
  \right|_{#2} 
  }}

\newcommand{\indicator}[1]{\mathbbm{1}_{\left\{ {#1} \right\} }}
\newcommand{\indic}[1]{\mathbbm{1}_{ {#1} }}
\newcommand{\E}[1]{\mathbb{E}\left[ {#1} \right] }
\newcommand{\Prob}[1]{\PP \left[ {#1} \right] }

\newcommand{\bzeta}{\boldsymbol{\zeta}}

\newcommand{\tbzeta}{\boldsymbol{\tilde{\zeta}}}
\newcommand{\tbxi}{\boldsymbol{\tilde{\xi}}}

\newcommand{\bxi}{\boldsymbol{\xi}}

\theoremstyle{plain}
\newtheorem{Theorem}{Theorem}[section]
\newtheorem{Proposition}[Theorem]{Proposition}

\newtheorem{Corollary}[Theorem]{Corollary}

\newtheorem{Definition}[Theorem]{Definition}

\theoremstyle{remark}
\newtheorem{Remark}[Theorem]{Remark}

\numberwithin{equation}{section}
\numberwithin{figure}{section}

\usepackage[
bookmarks,
bookmarksopen=true,
pdftitle={},
pdfauthor={Jan Obloj, Peter Spoida, Nizar Touzi},
pdfcreator={Jan Obloj, Peter Spoida, Nizar Touzi},
pdfsubject={},
pdfkeywords={},
colorlinks=false,
anchorcolor=black,
filecolor=magenta, 
menucolor=red, 
linkcolor=black, 
plainpages=false,
pdfpagelabels,
hypertexnames=false,
linktocpage 
]{hyperref}

\title{
Martingale Inequalities for the Maximum via Pathwise Arguments
\thanks{Jan Ob\l\'oj thankfully acknowledges support from the ERC Starting Grant {\sc RobustFinMath} 335421, the Oxford-Man Institute of Quantitative Finance and St John's College in Oxford. Peter Spoida gratefully acknowledges scholarships from the Oxford-Man Institute of Quantitative Finance and the DAAD. Nizar Touzi gratefully acknowledges financial support from the ERC Advanced Grant 321111 ROFIRM, the Chair {\it Financial Risks} of the {\it Risk Foundation} sponsored by Soci\'et\'e G\'en\'erale, and the Chair {\it Finance and Sustainable Development} sponsored by EDF and CA-CIB.}
}

\author{Jan Ob\l{}\'{o}j\thanks{University of Oxford, Mathematical Institute, the Oxford-Man Institute of Quantitative Finance and St John's College, Jan.Obloj@maths.ox.ac.uk} 
        \and Peter Spoida\thanks{University of Oxford, Mathematical Institute and the Oxford-Man Institute of Quantitative Finance, Peter.Spoida@maths.ox.ac.uk} 
        \and Nizar Touzi\thanks{Ecole Polytechnique Paris, Centre de Math\'ematiques Appliqu\'ees, nizar.touzi@polytechnique.edu}}

\date{\today}

\begin{document}

\maketitle

\begin{abstract}

We study a class of martingale inequalities involving the running maximum process. 
They are derived from pathwise inequalities introduced by \citet{Touzi_maxmax} and provide an upper bound on the expectation of a function of the running maximum in terms of marginal distributions at $n$ intermediate time points. The class of inequalities is rich and we show that in general no inequality is \emph{uniformly sharp} -- for any two inequalities we specify martingales such that one or the other inequality is sharper.
We then use our inequalities to recover Doob's $L^p$ inequalities. For $p\leq 1$ we obtain new, or refined, inequalities. 

\end{abstract}

\section{Introduction}

In this article we study certain martingale inequalities for the terminal maximum of a stochastic process. We thus contribute to a research area with a long and rich history. In seminal contributions, Blackwell and Dubins \cite{BlackwellDubins:63}, Dubins and Gilat \cite{DubinsGilat:78} and Az\'ema and Yor \cite{AzemaYor:79,AzemaYor:79b} showed that the distribution of the maximum $\bar{X}_T:=\sup_{t\leq T} X_t$ of a martingale $(X_t)$ is bounded above, in stochastic order, by the so called Hardy-Littlewood transform of the distribution of $X_T$, and the bound is attained. This led to series of studies on the possible distributions of $(X_T,\bar{X}_T)$, see Carraro, El Karoui and Ob\l\'oj \cite{CarraroElKarouiObloj:09} for a discussion and further references. More recently, such problems appeared very naturally within the field of mathematical finance. The original result was extended to the case of a non trivial starting law in Hobson \cite{Hobson:98b} and to the case of a fixed intermediate law in Brown, Hobson and Rogers \cite{Brown98themaximum}.

The novelty of our study here, as compared with the works mentioned above, is that we look at inequalities which use the information about the process at $n$ intermediate time points. One of our goals is to understand how the bound induced by these more elaborate inequalities compares to simpler inequalities which do not use information about the process at intermediate time points.
We show that in our context these bounds can be both, better or worse.
We also note that knowledge of intermediate moments does not induce a necessarily tighter bound in Doob's $L^p$-inequalities. 

Throughout, we emphasise the simplicity of our arguments, which are all elementary.
This is illustrated in Section \ref{sec:doob_section} where we obtain amongst others the sharp versions of Doob's $L^p$-inequalities for all $p>0$. While the case $p\geq 1$ is already known in the literature, our Doob's $L^p$-inequality in the case $p\in(0,1)$ appears new.

The idea of deriving martingale inequalities from pathwise inequalities is already present in work on robust pricing and hedging by \citet{Hobson:98b}.
Other authors have used pathwise arguments to derive martingale inequalities, e.g.
Doob's inequalities are considered by \citet{TrajectorialDoob} and \citet{OblojYor:06}.
The Burkholder-Davis-Gundy inequality is rediscovered with pathwise arguments by \citet{Beiglboeck:arXiv1305.6188}. 
In this context we also refer to \citet{CoxWang:11} and \citet{CoxPeskir:12} whose pathwise inequalities relate a process and time.
In a similar spirit, bounds for local time are obtained by \citet{MR2462552}.
\citet{Beiglbock:arXiv1401.4698} look at general martingale inequalities and explain how they can be obtained from deterministic inequalities. This approach builds on the so-called Burkholder's method, a classical tool in probability used to construct sharp martingale inequalities, see Os\c{e}kowski \cite[Chp.~2]{Osekowski:2012wh} for a detailed discussion.

In a discrete time and quasi-sure setup, the results of \citet{NB13:5} can be seen as general theoretical underpinning of many ideas we present here in the special case of martingale inequalities involving the running maximum.

\paragraph*{Organization of the article} 
In the Section \ref{sec:main_result} we state and prove our main result.
In Section \ref{sec:doob_section} we specialise our inequalities and demonstrate how they can be used to derive, amongst others, Doob's inequalities. We also investigate in which sense our martingale inequalities can provide sharper versions of Doob's inequalities.

\subsection{Preliminaries}

We assume that a filtered probability space $(\Omega,\mathcal{F}, (\mathcal{F}_t),\PP)$ is fixed which supports a standard real-valued Brownian motion $B$ with some initial value $X_0\in \R$. 
We will typically use $X=(X_t)$ to denote a (sub/super) martingale and, unless otherwise specified, we always mean this with respect to $X$'s natural filtration. Throughout, we fix arbitrary times $0=t_0\leq t_1\leq t_2\leq \ldots\leq t_n=:T$.

Before we proceed to the main result, we recall a remarkable pathwise inequality from \citet{Touzi_maxmax}. The version we give below appears in the proof of Proposition 3.1 in \cite{Touzi_maxmax} and is best suited to our present context. \newpage
\begin{Proposition}[Proposition 3.1 of \citet{Touzi_maxmax}]
\label{lem:Trajectorial_Inequality_Ordered_Case}
Let $\omega$ be a c\`{a}dl\`{a}g path and denote $\bar{\omega}_t:=\sup_{0 \leq s \leq t}{\omega_s}$. Then, for $m \geq \omega_0$ and $\zeta_1 \leq \dots \leq \zeta_n < m $:
\begin{alignat}{3}
 \indicator{\bar{\omega}_{t_n} \geq m} 
 \leq  
  \Upsilon_n(\omega, m, \bzeta ) := & &&\sum_{i=1}^{n}
 \left( 
 \frac{(\omega_{t_i}-\zeta_i)^+}
      {m-\zeta_i}
 +\indicator{\bar{\omega}_{t_{i-1}}<m \leq \bar{\omega}_{t_{i}}} \frac{m-\omega_{t_i}}
      {m-\zeta_i}
 \right)                    
 \label{eq:thm_Trajectorial_Inequality_Ordered_Case_1}\\
 &-&&\sum_{i=1}^{n-1}
 \left( 
 \frac{(\omega_{t_i}-\zeta_{i+1})^+}
      {m-\zeta_{i+1}} 
 +\indicator{m \leq \bar{\omega}_{t_{i}},
             \zeta_{i+1}\le\omega_{t_i}}  
  \frac{\omega_{t_{i+1}}-\omega_{t_i}}
       {m-\zeta_{i+1}} 
 \right).
\nonumber
\end{alignat}
\end{Proposition}

Next, we recall a process with some special structure in view of \eqref{eq:thm_Trajectorial_Inequality_Ordered_Case_1}. 
This process has been analysed in more detail by \citet{OblSp13}. 
 
\begin{Definition}[Iterated \AY Type Embedding]
\label{def:iterated_AY}
Let $\xi_1,\dots,\xi_n$ be non-decreasing functions and denote $\bar{B}_t:= \sup_{u\leq t} B_u$.
Set $\tau \equiv 0$ and for $i=1,\dots,n$ define  
\begin{numcases}{\tau_i:=}
\inf\left\{ t \geq \tau_{i-1}:B_t \leq \xi_{i}(\bar{B}_t) \right\} & \text{if $B_{\tau_{i-1}}> \xi_{i}(\bar{B}_{\tau_{i-1}}),$} \label{eq:embedding1} \\
\tau_{i-1} & \text{otherwise.}
\label{eq:embedding3}
\end{numcases} 
A continuous martingale $X$ is called an \iAY type embedding based on $\bxi = (\xi_1,\dots,\xi_n)$ if 
\begin{align}
(X_{t_i},\bar{X}_{t_i}) = (B_{\tau_i},\bar{B}_{\tau_i})\ a.s. \qquad \text{for $i=0,\dots,n$.}
\label{eq:changeoftimeAY}
\end{align}
\end{Definition} 

Note that $X$ being a martingale implies that $B_{\tau_i}$ are integrable and all have mean $X_0$. This then implies, by minimality of $\tau_i$, that $(B_{t\land \tau_n}, t\geq 0)$ is a uniformly integrable martingale. If the latter is true then an example of an \iAY type embedding is obtained by taking
$$X_t:=B_{\tau_i\land \left (\tau_{i-1}\lor \frac{t-t_{i-1}}{t_{i} - t }\right)},\quad \textrm{for}\quad t_{i-1}<t\leq t_i,\ i=1,\ldots,n.
$$
Finally, we recall a version of Lemma 4.1 from \citet{Touzi_maxmax}.
\begin{Proposition}[Pathwise Equality]
\label{prop:Pathwise Equality}
Let $\bxi = (\xi_1,\dots,\xi_n)$ be non-decreasing right-continuous functions and let $X$ be an \iAY embedding based on $\bxi$. 
Then, for any $m > X_0$ with $\xi_n(m) < m$, $X$ achieves equality in \eqref{eq:thm_Trajectorial_Inequality_Ordered_Case_1}, i.e.
\begin{align}\label{eq:pathwise_equality}
\indicator{ \bar{X}_{t_n} \geq m } 
= \Upsilon_n \big(X,m,\bzeta(m) \big)\quad a.s.,
\end{align}
where
\begin{align}
\zeta_i(m) = \min_{j\geq i} \xi_j(m), \qquad i=1,\dots,n.
\label{eq:relation_zeta_xi_pathwise_equality_result}
\end{align} 
\end{Proposition}

We note that if we work on the canonical space of continuous functions then \eqref{eq:pathwise_equality} holds pathwise and not only a.s. We also note that the assumption that $X$ is an \iAY type embedding, or that $(B_{\tau_n\land t})$ is a uniformly integrable martingale, may be relaxed as long as $X$ satisfies \eqref{eq:changeoftimeAY}.

\section{Main Result}
\label{sec:main_result}

In our main results, we obtain and compare inequalities for \cadlag submartingales which are directly implied by Proposition \ref{lem:Trajectorial_Inequality_Ordered_Case}.

\subsection{Main Result -- Part 1}
\label{subsec:part_1}

In the first part of our main result we devise a general martingale inequality for $\E{\phi( \bar{X}_{T} )}$ and prove that it is attained under some conditions.

Define 
\begin{equation}
\begin{alignedat}{2}
\mathscr{Z} := \Big\{ \bzeta = ( \zeta_1, \dots, \zeta_n ) \ : \ 
&\zeta_i: (X_0,\infty) \to \R \ \text{ is right-continuous},\\
&\zeta_1(m) \leq \dots \leq \zeta_n(m)<m, \quad n \in \N \Big\}.
\label{eq:zeta_minimal_condition}
\end{alignedat}
\end{equation}

In order to ensure that the expectations we consider are finite we will occasionally need the technical condition that 
\begin{equation}
\exists \alpha>0 \textrm{ s.t. }\liminf_{m \to \infty} \frac{\zeta_1(m)}{\alpha m }  \geq 1\textrm{ and }
\limsup_{m \to \infty} \frac{\phi(m)}{m^{\gamma}}=0 \quad \text{for some $\gamma < \frac{1}{1-\alpha}$.}
\label{eq:integrability}
\end{equation}

\begin{Theorem}[Main Result -- Part 1]
\label{thm:main_result_1}
Let $\phi$ a right-continuous non-decreasing function, and $\bzeta = ( \zeta_1, \dots, \zeta_n ) \in \mathscr{Z}$. Then,
\\
{\rm (i)} for all \cadlag submartingale $X$: 
\begin{alignat}{3}
\E{ \phi(\bar{X}_{T}) } 
\leq \mathrm{UB} \left( X, \phi, \bzeta \right) := 
\phi(X_0) +   \int_{(X_0,\infty)}\sum_{i=1}^n \E{\lambda^{\bzeta,m}_i(X_{t_i})} \dd \phi(m) 
\label{eq:martingale_inequality_general}
\end{alignat}
where
\begin{align}
\lambda^{\bzeta,m}_i(x) :=  \frac{\left( x-\zeta_i(m) \right)^+}{m-\zeta_i(m)} 
                                            - \frac{\left( x-\zeta_{i+1}(m) \right)^+}{m-\zeta_{i+1}(m)} \indicator{i<n} ,
\label{eq:lambda_opt} 
\end{align}
{\rm (ii)} if $\zeta_1$ is non-decreasing and satisfies, together with $\phi$, the condition \eqref{eq:integrability}, there exists a continuous martingale which achieves equality in \eqref{eq:martingale_inequality_general}.
\end{Theorem}

\begin{Remark}[Optimization over $\bzeta$]
\label{rem:Optimization over zeta}
If $X$ and $t_1,\dots,t_n$ are fixed we can optimize \eqref{eq:martingale_inequality_general} over $\bzeta \in \mathscr{Z}$ to obtain a minimizer $\bzeta^{\star}$. 
Clearly, more intermediate points $t_i$ in \eqref{eq:martingale_inequality_general} can only improve the bound for this particular process $X$. 
However, only for very special processes (e.g. the \iAY type embedding) there is hope that \eqref{eq:martingale_inequality_general} will hold with equality.
This is, loosely speaking, because a finite number of intermediate marginal law constraints does not, in general, determine uniquely the law of the maximum at terminal time $t_n$.
\end{Remark}

\begin{proof}[Proof of Theorem \ref{thm:main_result_1}]
Equation \eqref{eq:martingale_inequality_general} follows from \eqref{eq:thm_Trajectorial_Inequality_Ordered_Case_1} by taking expectations and integrating against $\dd \phi$. Note that for a fixed $m$, $\E{|\lambda^{\bzeta,m}_i(X_{t_i})|}<\infty$ for $i=1,\ldots, n$, since $\E{|X_{t_i}|}<\infty$ by the submartingale property.

If $\zeta_1$ is non-decreasing and $\zeta_1(m) \geq \alpha m$ for $m$ large, $\alpha > 0$, we define $X$ by 
\begin{align*}
X_t = 
\begin{cases} 
	B_{ \frac{t}{t_1 - t } \wedge \tau_{\zeta_1} } &\mbox{if } t < t_1, \\
	B_{\tau_{\zeta_1}} & \mbox{if } t \geq t_1. 
\end{cases}
\end{align*}
where $B$ is a Brownian motion, $B_0 = X_0$, and $\tau_{\zeta_1} := \inf \left\{ u>0 \ : \ B_u \leq \zeta_1(\bar{B}_u) \right\}$. 
$X$ is a uniformly integrable martingale by similar arguments as in the proof of \citet[Prop. 3.5]{OblSp13}.
Then, one readily verifies together with Proposition \ref{prop:Pathwise Equality} that 
\begin{align*}
\Upsilon_n(X,m,\bzeta) = \Upsilon_1(X,m,\bzeta) = \indicator{ \bar{X}_{t_1} \geq m } = \indicator{ \bar{X}_{t_n} \geq m }.
\end{align*}
Condition \eqref{eq:integrability} ensures that $\E{ \phi(\bar{X}_{t_n}) } < \infty$ because then by excursion theoretical results, cf. e.g. \citet{rogers89g}, we compute
\begin{align*}
\Prob{ \bar{X}_{t_n} \geq y } 
&= \exp\left( - \int_{(X_0,y]} \frac{1}{z - \zeta_1(z)} \dd z \right)
\leq \mathrm{const} \cdot \exp \left( - \int_{(1,y]} \frac{1}{z - \alpha z } \dd z \right) \\
&= \mathrm{const} \cdot y^{-\frac{1}{1-\alpha}}
\end{align*}
for large $y$. 
Now the claim follows from
\begin{align*}
\E{ \phi(\bar{X}_{t_n}) } &= \phi(X_0) + \int_{(X_0,\infty)} \E{ \indicator{ \bar{X}_{t_n} \geq m } } \dd \phi(m) \\
						   &=\phi(X_0) + \int_{(X_0,\infty)} \mathrm{UB}\left( X, \indic{[m,\infty)}, \bzeta \right) \dd \phi(m) \\
						   &=\phi(X_0) + \mathrm{UB}\left( X, \phi, \bzeta \right)
\end{align*}
where we applied Fubini's theorem.
\end{proof}

\subsection{Main Result -- Part 2}
\label{subsec:part_2}

As mentioned in the introduction, the novelty of our martingale inequality from Theorem \ref{thm:main_result_1} is that it uses information about the process at intermediate times.
The second part of our main result sheds light on the question whether this information gives more accurate bounds than e.g. in the case when no information about the process at intermediate times is used.
In short, the answer is negative, i.e. we demonstrate that for a large class of $\tbzeta$'s there is no \enquote{universally better} choice of $\bzeta$ in the sense that it yields a tighter bound in the class of inequalities for $\E{ \phi(\bar{X}_{T}) }$ from Theorem \ref{thm:main_result_1}.

To avoid elaborate technicalities, we impose additional conditions on $\bzeta \in \mathscr{Z}$ and $\phi$ below. Many of these conditions could be relaxed to obtain a slightly stronger, albeit more involved, statement in Theorem \ref{thm:main_result_2}. We define 
\begin{align}
\mathscr{Z}^{\mathrm{cts}} := \Big\{ \bzeta \in \mathscr{Z} \ : \ \ & \text{$\bzeta$ are continuous} \Big \}
\end{align}
and
\begin{equation}
\begin{split}
\tilde{ \mathscr{Z} } := \Big\{ \bzeta \in \mathscr{Z}^{\mathrm{cts}} \ : \ \ & \text{$\bzeta$ are strictly increasing and} \\
&\liminf_{m \to \infty} \zeta_1(m)/\alpha m \geq 1,~\mbox{for some}~\alpha>0,  \\
& \zeta_1 = \dots = \zeta_n~\mbox{on}~(X_0,X_0 + \epsilon),~\mbox{for some}~\epsilon>0 \Big\}. 
\end{split}
\label{eq:definition_Z}
\end{equation} 
Before we proceed, we want to argue that the set $\tilde{\mathscr{Z}}$ arises quite naturally.
In the setting of Remark \ref{rem:Optimization over zeta}, if $X$ is a martingale such that its marginal laws
\begin{align*}
\mu_1:=\mathcal{L}\left( X_{t_1} \right), \ \ \ \dots  , \ \ \ \mu_n := \mathcal{L}\left( X_{t_n} \right)
\end{align*}
satisfy Assumption \text{$\circledast$} of \citet{OblSp13}, $\int ( x- \zeta )^+ \mu_i(\dd x) < \int ( x- \zeta )^+ \mu_{i+1}(\dd x)$ for all $\zeta$ in the interior of the support of $\mu_{i+1}$ and their barycenter functions satisfy the mean residual value property of \citet{MadanYor} close to $X_0$ and have no atoms at the left end of support, then the optimization over $\bzeta$ as described in Remark \ref{rem:Optimization over zeta} yields a unique $\tbzeta^{\star} \in \tilde{\mathscr{Z}}$.
Hence, the set of these $\tilde{\mathscr{Z}}$ seems to be a \enquote{good candidate set} for $\bzeta$'s to be used in Theorem \ref{thm:main_result_1}.

The statement of the Theorem \ref{thm:main_result_2} concerns the negative orthant of $\mathscr{Z}^{\mathrm{cts}}$, 
\begin{equation}
\begin{alignedat}{2}
\mathscr{Z}^{\mathrm{cts}}_{-}( \phi, \tbzeta ) := \Big\{ \bzeta \in \mathscr{Z}^{\mathrm{cts}} \ : \ &\mathrm{UB} \left(X, \phi, \bzeta \right) \leq \mathrm{UB} \left(X, \phi, \tbzeta \right) \text{ for all \cadlag} \\ 
&
  \text{submartingales $X$ and $<$ for at least one $X$} \Big\}.
\label{eq:orthant}
\end{alignedat}
\end{equation}

\begin{Theorem}[Main Result -- Part 2]
\label{thm:main_result_2}
Let $\phi$ be a right-continuous, strictly increasing function. Then, for $\tbzeta \in \tilde{\mathscr{Z}}$ such that \eqref{eq:integrability} holds we have
\begin{align}
\mathscr{Z}^{\mathrm{cts}}_{-}( \phi, \tbzeta ) = \emptyset.
\end{align}
\end{Theorem}

The above result essentially says that no martingale inequality in \eqref{eq:martingale_inequality_general} is universally better than another one. For any choice $\tbzeta \in \tilde{\mathscr{Z}}$, the corresponding martingale inequality \eqref{eq:martingale_inequality_general} can not be strictly improved by some other choice of $\bzeta \in \mathscr{Z}^{\mathrm{cts}}$, i.e.\ no other $\bzeta$ would lead to a better upper bound for all submartingales and strictly better for some submartingale. The key ingredient to prove this statement is isolated in the following Proposition.
\begin{Proposition}[Positive Error]
\label{prop:positive_hedging_error}
Let $\tbzeta \in \tilde{\mathscr{Z}}$ and $\bzeta \in \mathscr{Z}^{\mathrm{cts}}$ satisfy $\tbzeta \neq \bzeta$. Then there exists a non-empty interval $(m_1,m_2) \subseteq (X_0,\infty)$ such that 
\begin{eqnarray*}
\mathrm{UB} \left(X,\indic{[m,\infty)},\tbzeta \right) 
< 
\mathrm{UB} \left( X,\indic{[m,\infty)},\bzeta \right)
&\mbox{for all}&
m\in(m_1,m_2),
\label{eq:pos_hedging_error}
\end{eqnarray*}
where $X$ is an \iAY type embedding based on some $\tbxi$.
\end{Proposition}

\begin{proof}
To each $\tbzeta \in \tilde{\mathscr{Z}}$ we can associate a non-decreasing and continuous stopping boundary $\tbxi$  which satisfies
\begin{subequations}
\begin{alignat}{3}
& \tilde{\xi}_n(m) < \dots < \tilde{\xi}_1(m)<m \qquad &&\forall m \in (X_0,X_0+\epsilon),
\label{eq:relation_zeta_xi_1} \\
&  \tbxi(m) = \tbzeta(m) \qquad && \forall m \geq X_0 + \epsilon,
\label{eq:relation_zeta_xi_2}  
\end{alignat}
\end{subequations}
for some $\epsilon>0$, and hence
\begin{align}
\tilde{\zeta}_i(m) = \min_{j\geq i} \tilde{\xi}_j(m) \qquad \forall m > X_0.
\label{eq:relation_zeta_xi}  
\end{align}
Fix such a $\tbxi$ and let $X$ be an \iAY type embedding based on this $\tbxi$.
Let $j \geq 1$.
Using the notation of Definition \ref{def:iterated_AY}, it follows by monotonicity of $\tbxi$, \eqref{eq:relation_zeta_xi_2} and \eqref{eq:relation_zeta_xi} that on the set $\{ B_{\tau_j} = \tilde{\xi}_j( \bar{B}_{\tau_j} ), \ \bar{B}_{\tau_j} \geq X_0 + \epsilon \}$ we have $B_{\tau_j} = \tilde{\xi}_{j}(\bar{B}_{\tau_j}) \leq \tilde{\xi}_{j+1}( \bar{B}_{\tau_j} )$.
Therefore, the condition of \eqref{eq:embedding1} in the definition of the \iAY type embedding is not satisfied and hence $\tau_{j+1} = \tau_j$.
Consequently,
\begin{equation}
\begin{split}
&X_{t_j} = X_{t_{j+1}} = \dots = X_{t_n} \quad \text{and} \quad \bar{X}_{t_j} = \bar{X}_{t_{j+1}} = \dots = \bar{X}_{t_n} \\
&\text{on the set $\left\{ X_{t_j} =  \tilde{\xi}_j(\bar{X}_{t_j}), \ \bar{X}_{t_j} \geq X_0 + \epsilon \right\}$}
\end{split}
\label{eq:iAY_embedding_constant}
\end{equation}
for all $j \geq 1$.

Take $1 \leq j \leq n$.
Denote $\chi := \max \{ k \leq n : \exists t \leq H_{X_0 + \epsilon} \text{ s.t. } B_{t} \leq \tilde{\xi}_{k} (\bar{B}_t ) \} \vee 0$, where $H_x := \inf\{ u>0:B_u=x \}$ and $\mathcal{H} := \{ \chi = j-1, \ H_{X_0 + \epsilon} < \infty \}$. 
By \eqref{eq:relation_zeta_xi_1} we have $\Prob{ \mathcal{H} } > 0$.
Further, by using $\tilde{\zeta}_1(m) \leq \dots \leq \tilde{\zeta}_n(m) < m$ we conclude by the properties of Brownian motion that $\Prob{ \mathcal{H}  \cap \{ \bar{B}_{\tau_j} \in \mathcal{O} \} } > 0$ for $\mathcal{O} \subseteq (X_0 + \epsilon,\infty)$ an open set.
Relabelling and using \eqref{eq:relation_zeta_xi_2} yields
\begin{align}
\Prob{ X_{t_j} =  \tilde{\zeta}_j(\bar{X}_{t_j}),  \bar{X}_{t_j} \in \mathcal{O},  \bar{X}_{t_{j-1}} < X_0 + \epsilon } > 0  \ \ \text{$\forall$ open $\mathcal{O} \subseteq (X_0 + \epsilon,\infty)$.}
\label{eq:pos_prob_error_set}
\end{align}

By $\tbzeta \neq \bzeta$ either Case A or Case B below holds (possibly by changing $\epsilon$ above).
In our arguments we refer to the proof of the pathwise inequality of Proposition \ref{lem:Trajectorial_Inequality_Ordered_Case} given by \citet{Touzi_maxmax} and argue that certain inequalities in this proof become strict.

{\bf Case A}: $\exists \mathcal{O} = (m_1,m_2) \subseteq (X_0 + \epsilon,\infty)$ and $j \leq n$ such that $\tilde{\zeta}_j(m_1) > \zeta_j(m_2)$. 
Take $m > m_2$.
Then on $\left\{ X_{t_j} =  \tilde{\zeta}_j(\bar{X}_{t_j}), \  \bar{X}_{t_j} \in \mathcal{O} \right\}$ we have almost surely
\begin{alignat*}{3}
\Upsilon_n( X, m, \bzeta ) \stackrel[]{ \eqref{eq:iAY_embedding_constant} }{=} \Upsilon_j( X, m, \bzeta ) 
>0 = \indicator{ m \leq \bar{X}_{t_j} } \stackrel[]{ \eqref{eq:iAY_embedding_constant} }{=} \indicator{ m \leq \bar{X}_{t_n} } \stackrel[]{\text{Prop. } \ref{prop:Pathwise Equality} }{=} \Upsilon_n( X, m, \tbzeta )
\end{alignat*}
where the strict inequality holds by noting that $( X_{t_j} - \zeta_j(m) )^+ > 0 $ for all $m \in (m_1,m_2)$ on the above set and then directly verifying that the second inequality of equation (4.3) of \citet{Touzi_maxmax} applied with $\bzeta$ and $X$ is strict. 

{\bf Case B}: $\exists \mathcal{O} = (m_1,m_2) \subseteq (X_0 + \epsilon,\infty)$ and $j \leq n$ such that $\tilde{\zeta}_j(m_2) < \zeta_j(m_1)$.
Take $m \in \mathcal{O}$.
Then on $\left\{ X_{t_j} =  \tilde{\zeta}_j(\bar{X}_{t_j}),  \  \bar{X}_{t_j} \in \mathcal{O} \cap (m,\infty), \ \bar{X}_{t_{j-1}} < X_0 + \epsilon \right\}$ we have almost surely
\begin{align*}
\Upsilon_n( X, m, \bzeta ) \stackrel[]{ \eqref{eq:iAY_embedding_constant} }{=} \Upsilon_j( X, m, \bzeta )
> 1 = \indicator{ m \leq \bar{X}_{t_j} } = \indicator{ m \leq \bar{X}_{t_n} } \stackrel[]{\text{Prop. } \ref{prop:Pathwise Equality} }{=} \Upsilon_n( X, m, \tbzeta )
\end{align*}
where the strict inequality holds by observing that the last inequality in equation (4.3) of \citet{Touzi_maxmax} applied with $\bzeta$ and $X$ is strict because $(X_{j} - \zeta_j(m))^+ = 0 > X_{j} - \zeta_j(m)$ for all $m \in \mathcal{O}$ on the above set.

Combining, in both cases A and B the claim \eqref{eq:pos_hedging_error} follows from \eqref{eq:pos_prob_error_set}.
\end{proof}

\begin{proof}[Proof of Theorem \ref{thm:main_result_2}]
Take $\bzeta \in \mathscr{Z}^{\mathrm{cts}}$ such that strict inequality holds for one submartingale in the definition of $\mathscr{Z}_{-}^{\mathrm{cts}}$, see \eqref{eq:orthant}. We must have $\bzeta \neq \tbzeta$.

As in the proof of Proposition \ref{prop:positive_hedging_error} we choose a $\tbxi$ such that \eqref{eq:relation_zeta_xi_1}--\eqref{eq:relation_zeta_xi_2}, \eqref{eq:relation_zeta_xi} hold and let $X$ be an \iAY type embedding based on this $\tbxi$.
Propositions \ref{lem:Trajectorial_Inequality_Ordered_Case} and \ref{prop:Pathwise Equality} yield
\begin{align*}
\E{ \indic{[m,\infty)}( \bar{X}_{t_n} ) } 
= 
\mathrm{UB} \left(X,\indic{[m,\infty)},\tbzeta \right) 
\leq 
\mathrm{UB}\left(X,\indic{[m,\infty)},\bzeta \right) \qquad \forall m > X_0
\end{align*}
and by Proposition \ref{prop:positive_hedging_error} 
\begin{align*}
\mathrm{UB} \left( X,\indic{[m,\infty)},\tbzeta \right) < \mathrm{UB}\left(X,\indic{[m,\infty)},\bzeta \right)
\end{align*}
for all $m \in \mathcal{O}$ where $\mathcal{O} \subseteq (X_0,\infty)$ is some open set. 
Now the claim follows as in the proof of Theorem \ref{thm:main_result_1}.
\end{proof}

\begin{Remark}
In the setting of Theorem \ref{thm:main_result_2} let $\tbzeta^1, \tbzeta^2 \in \tilde{\mathscr{Z}}, \ \tbzeta^1 \neq \tbzeta^2$, 
and assume that \eqref{eq:integrability} holds for $(\phi,\tbzeta^1)$ and $(\phi,\tbzeta^2)$.
Then there exist martingales $X^{1}$ and $X^{2}$ such that
\begin{align*}
\mathrm{UB}\left( X^1, \phi, \tbzeta^1 \right) &< \mathrm{UB} \left( X^1, \phi, \tbzeta^2 \right), \\
\mathrm{UB}\left( X^2, \phi, \tbzeta^1 \right) &> \mathrm{UB} \left( X^2, \phi, \tbzeta^2 \right).
\end{align*}
This follows by essentially reversing the roles of $\tbzeta^{1}$ and $\tbzeta^{2}$ in the proof of Theorem \ref{thm:main_result_2}.
\end{Remark}

\section{Doob's Inequalities}
\label{sec:doob_section}

In this section we demonstrate how Theorem \ref{thm:main_result_1} can be used to derive Doob's inequalities. 
Further, we investigate in which sense there is an improvement to Doob's inequalities.

Related work on pathwise interpretations of Doob's inequalities can be found in \citet{TrajectorialDoob} and \citet{OblojYor:06}.
\citet[Section 4]{peskir1998} derives Doob's inequalities and shows that the constants he obtains are optimal.
We now give an alternative proof of these statements, and we provide new sharp inequalities for the case $p<1$.

\subsection{Doob's \texorpdfstring{$L^p$}{Lp}-Inequalities, \texorpdfstring{$p>1$}{p>1} }
\label{subsec:L_p}

Using a special case of Theorem \ref{thm:main_result_1} we obtain an improvement to Doob's inequalities.
Denote $\ \mathrm{pow}^p(m) = m^p, \ \ \zeta_{\alpha}(m) := \alpha m$.
\begin{Proposition}[Doob's $L^p$-Inequalities, $p>1$]
\label{prop:doob_Lp}
Let $(X_t)_{t \leq T}$ be a non-negative \cadlag submartingale. 
\begin{itemize}
	\item[{\rm (i)}]  Then,
\begin{subequations}
\begin{align}
\E{ \bar{X}_T^p } 
	& \quad \leq \quad  \mathrm{UB} \left( X, \mathrm{pow}^p, \zeta_{\frac{p-1}{p}} \right) \label{eq:Doob_improvement_1} \\
	&\quad \leq \quad  \left( \frac{p}{p-1} \right)^p \E{ X_T^p } - \frac{p}{p-1} X_0^p.
\label{eq:doob_Lp} 
\end{align}
\end{subequations}
	\item[{\rm (ii)}] For every $\epsilon>0$, there exists a martingale $X$ such that
\begin{align}
0 \leq \left( \frac{p}{p-1} \right)^p   \E{ X_{T}^{p} } - \frac{p}{p-1} X_0^p - \E{ \bar{X}_T^p } < \epsilon.
\label{eq:sharpness_doob_Lp}
\end{align}
	\item[{\rm (iii)}]  The inequality in \eqref{eq:doob_Lp} is strict if and only if either holds:
\begin{subequations}
\begin{align}
&\text{$\E{ \bar{X}_T^p }<\infty$ and $X_T < \frac{p-1}{p} X_0$ with positive probability.}
\label{eq:strict_error_Doob_Lp_1} \\
&\text{$\E{ \bar{X}_T^p }<\infty$ and $X$ is a strict submartingale.}
\label{eq:strict_error_Doob_Lp_2} 
\end{align}
\end{subequations}
\end{itemize}
\end{Proposition}

\begin{proof}
Let us first prove \eqref{eq:Doob_improvement_1} and \eqref{eq:doob_Lp}.
If $\E{ X_T ^p } = \infty$ there is nothing to show. 
In the other case, equation \eqref{eq:Doob_improvement_1} follows from Theorem \ref{thm:main_result_1} applied with $n=1, \ \phi(y) = \mathrm{pow}^p(y) = y^p$ and $\zeta_1 = \zeta_{\frac{p-1}{p}}$. To justify this choice of $\zeta_1$ and to simplify further the upper bound we start with a more general $\zeta_1 = \zeta_{\alpha}, \ \alpha < 1$ and compute
\begin{alignat}{3}
&&\quad&    \E{ \bar{X}_T^p } - X_0^p 
		  \ \leq \ \mathrm{UB} \left( X, \mathrm{pow}^p, \zeta_{\alpha} \right) - X_0^p \ = \ \E{ \int_{X_0}^{\infty} p y^{p-1} \frac{ ( X_T - \alpha y )^+ }{ y-\alpha y } \dd y } 
\nonumber \\ 
		  = &&& \E{ \int_{X_0}^{\frac{X_T}{\alpha} \vee X_0}  p y^{p-1} \frac{ X_T - \alpha y  }{ y-\alpha y }  \dd y } \quad \leq \quad \E{ \int_{X_0}^{\frac{X_T}{\alpha} }  p y^{p-1} \frac{ X_T - \alpha y  }{ y-\alpha y }  \dd y  } 
\nonumber \\ 		  
		  = &&& \frac{p}{p-1}\frac{1}{1-\alpha} \E{ \left\{ \left( \frac{X_T}{\alpha} \right)^{p-1} - X_0^{p-1} \right\} X_T } - \frac{\alpha}{1-\alpha} \E{ \left( \frac{X_T}{\alpha} \right)^{p}  - X_0^{p} }   
\nonumber \\ 		  
		  \leq &&& 	\frac{1}{p-1}\frac{1}{(1-\alpha)\alpha^{p-1}}	 \E{ X_T^{p}  } - \frac{p-\alpha (p-1)}{(p-1)(1-\alpha)} X_0^p
\label{eq:estimations_Doob_Lp}
\end{alignat}
where we used Fubini in the first equality and the submartingale property of $X$ in the last inequality.
We note that the function $\alpha \mapsto \frac{1}{(1-\alpha) \alpha^{p-1}}$ attains its minimum at $\alpha^{\star} = \frac{p-1}{p}$. Plugging $\alpha = \alpha^{\star}$ into the above yields \eqref{eq:doob_Lp}.

We turn to the proof that Doob's $L^p$-inequality is attained asymptotically in the sense of \eqref{eq:sharpness_doob_Lp}, a fact which was also proven by \citet[Section 4]{peskir1998}. 
Let $X_0 > 0$, otherwise the claim is trivial.
Set $\alpha^{\star}=\frac{p-1}{p}$ and take $\alpha^{\star}< \alpha := \frac{p+\epsilon- 1}{p+\epsilon} < 1$.
Let $X_T = B_{\tau_{\alpha}}$ where $B$ is a Brownian motion stared at $X_0$ and  $\tau_{\alpha} := \inf\{ u>0:B_u \leq \alpha \bar{B}_u \}$.
Then by using excursion theoretical results, cf. e.g. \citet{rogers89g},
\begin{align*}
\Prob{ \bar{X}_T \geq y } = \exp\left( - \int_{X_0}^{y} \frac{1}{z-\alpha z} \dd z \right) = \left( \frac{y}{X_0} \right)^{ -\frac{1}{1-\alpha} }
\end{align*}
and then direct computation shows
\begin{align*}
\E{\bar{X}_T^p} = \frac{p+\epsilon}{\epsilon}X_0^p.
\end{align*}
By Doob's $L^p$-inequality,
\begin{align*}
\E{ \bar{X}_T^p } \leq 
\left( \frac{p}{p-1}\right)^p \E{X_T^p} - \frac{p}{p-1}X_0^p 
= & \left(\frac{\alpha}{\alpha^{\star}}\right)^p \E{ \bar{X}_T^p } - \frac{p}{p-1} X_0^p
\end{align*}
and one verifies
\begin{align*}
\left\{ \left(\frac{p}{p-1}\right)^p \cdot \left[ \frac{p+\epsilon - 1}{p+\epsilon} \right]^p - 1 \right\} \cdot \frac{p+\epsilon}{\epsilon}X_0^p \quad \xrightarrow[\quad\epsilon \downarrow 0\quad]{} \quad \frac{p}{p-1}X_0^p.
\end{align*}
This establishes the claim in \eqref{eq:sharpness_doob_Lp}.

Finally, we note that in the calculations \eqref{eq:estimations_Doob_Lp} which led to \eqref{eq:doob_Lp} there are three inequalities:
the first one comes from Theorem \ref{thm:main_result_1} and does not concern the claim regarding \eqref{eq:strict_error_Doob_Lp_1}--\eqref{eq:strict_error_Doob_Lp_2}.
The second one is clearly strict if and only if \eqref{eq:strict_error_Doob_Lp_1} holds.
The third one is clearly strict if and only if \eqref{eq:strict_error_Doob_Lp_2} holds.
\end{proof}

\begin{Remark}[Asymptotic Attainability]
For the martingales in {\rm (ii)} of Proposition \ref{prop:doob_Lp} we have
\begin{align*}
\mathrm{UB} \left( X, \mathrm{pow}^p, \zeta_{\frac{p-1}{p}} \right) 
= \left( \frac{p}{p-1} \right)^p \E{ X_T^p } - \frac{p}{p-1} X_0^p
\end{align*}
and $\E{ X_T^p } \to \infty$ as $\epsilon \to 0$.
\end{Remark}

\subsection{Doob's \texorpdfstring{$L^1$}{L1}-Inequality}
\label{subsec:L_1}

Using a special case of Theorem \ref{thm:main_result_1} we focus on Doob's $L\log L$ type inequalities. We recover here the classical constant $e/(e-1)$, see \eqref{eq:doob_L1} , with a refined structure on the inequality. A further improvement to the constant will be obtain in subsequent section in Corollary \ref{cor:NewL1}. Denote $\mathrm{id}(m)=m$, and 
\begin{align}
\underline{\zeta}_{\alpha}(m) := 
\begin{cases} 
	-\infty &\mbox{if } m<1, \\
	\alpha m & \mbox{if } m \geq 1. 
\end{cases}
\end{align}

\begin{Proposition}[Doob's $L^1$-Inequality]
\label{prop:doob_L1}
Let $(X_t)_{t \leq T}$ be a non-negative \cadlag submartingale. Then:
\begin{itemize}
	\item[{\rm (i)}] with $0 \log(0):= 0$ and $V(x):=x-x\log(x)$,
\begin{subequations}
\begin{align}
\E{ \bar{X}_T } 
	& \quad \leq \quad \mathrm{UB}\left( X, \mathrm{id}, \underline{\zeta}_{\frac{1}{e}} \right) 
\label{eq:Doob_improvement_2} \\
	& \quad \leq \quad \frac{e}{e-1} \Big( \E{ X_T \log\left( X_T \right) } + V(1\vee X_0) \Big).
\label{eq:doob_L1} 
\end{align}
\end{subequations}
	\item[{\rm (ii)}] in the case $X_0 \geq 1$ there exists a martingale which achieves equality in both, \eqref{eq:Doob_improvement_2} and \eqref{eq:doob_L1} and in the case $X_0 < 1$ there exists a submartingale which achieves equality in both, \eqref{eq:Doob_improvement_2} and \eqref{eq:doob_L1}.
	\item[{\rm (iii)}] the inequality in \eqref{eq:doob_L1} is strict if and only if either holds:
\begin{subequations}
\begin{align}
&\text{$\E{ \bar{X}_T }<\infty$ and $\bar{X}_T \geq 1, \ \ X_T < \frac{1}{e} X_0$ with positive probability,}
\label{eq:strict_error_Doob_L1_1}  \\	
&\text{$\E{ \bar{X}_T }<\infty$ and $\bar{X}_T \geq 1, \ \ \E{ X_T } > X_0 \vee 1$.}
\label{eq:strict_error_Doob_L1_2}  \\	
&\text{$\E{ \bar{X}_T }<\infty$ and $\bar{X}_T<1$ with positive probability.} 
\label{eq:strict_error_Doob_L1_3} 
\end{align}
\end{subequations}
\end{itemize}
\end{Proposition}

\begin{proof}
Let us first prove \eqref{eq:Doob_improvement_2} and \eqref{eq:doob_L1}.
If $\E{ \bar{X}_T } = \infty$ there is nothing to show. 
In the other case, equation \eqref{eq:Doob_improvement_2} follows from Theorem \ref{thm:main_result_1} applied with $n=1, \ \phi(y) = \mathrm{id}(y)=y$ and $\zeta_1 = \underline{\zeta}_{\frac{1}{e}}$.

In the case $X_0 \geq 1$ we further compute using $\zeta_1 = \underline{\zeta}_{\alpha}, \ \alpha < 1$,
\begin{alignat}{3}
	 &&\quad& \E{ \bar{X}_T } - X_0 \quad \leq \quad \mathrm{UB}\left( X, \mathrm{id}, \underline{\zeta}_{\alpha} \right) - X_0  \nonumber \\
= &&& \int_{X_0}^{ \frac{X_T}{\alpha} \vee X_0 } \frac{\E{  X_T -\alpha y  }}{ y- \alpha y } \dd y \quad  \leq  \quad  \E{ \int_{X_0}^{ \frac{X_T}{\alpha} }  \frac{ X_T - \alpha y }{ (1-\alpha) y } \dd y } \nonumber \\
= &&& \frac{\alpha}{1-\alpha}  \E{ \frac{X_T}{\alpha} \left\{ \log\left(\frac{X_T}{\alpha} \right) - \log(X_0) \right\} } - \frac{\alpha}{1-\alpha}   \E{ \frac{X_T}{\alpha} - X_0 }
\nonumber \\ 
 \stackrel{\mathclap{\alpha = 1/e}}{=} &&& \frac{e}{e-1} \E{ X_T \log\left(X_T \right) } - \frac{e}{e-1} \E{ X_T } \log(X_0) + \frac{1}{e-1}X_0   
\nonumber \\
 \leq &&& \frac{e}{e-1} \E{ (X_T) \log\left(X_T \right) } - \frac{e}{e-1} X_0 \log(X_0)  + \frac{1}{e-1}X_0.
\label{eq:calculations_Doob_L1_X_0_geq_1}
\end{alignat}
where the choice $\alpha = \frac{1}{e}$ gives a convenient cancellation and we used again that $X$ is a submartingale.
This is \eqref{eq:doob_L1} in the case $X_0 \geq 1$.

For the case $0<X_0<1$ we obtain from Proposition \ref{lem:Trajectorial_Inequality_Ordered_Case} for $n=1$, 
\begin{align*}
\Prob{ \bar{X}_T \geq y } \ \leq \ \inf_{\zeta<y} \frac{\E{ ( X_T - \zeta )^+ }}{ y-\zeta } \leq \frac{\E{ ( X_T - \alpha y )^+ }}{ y-\alpha y }
\end{align*}
for $\alpha < 1$ and therefore
\begin{equation}
\begin{split}
\E{ \bar{X}_T } - X_0 
\ = \ & \int_{X_0}^{\infty} \Prob{ \bar{X}_T \geq y } \dd y  
\ \leq \ (1-X_0) + \int_1^{\infty} \Prob{ \bar{X}_T \geq y } \dd y 
\\
\leq \ &(1-X_0) + \frac{e}{e-1} \E{ (X_T) \log\left(X_T \right) } + \frac{1}{e-1}
\end{split}
\label{eq:calculations_Doob_L1_X_0_leq_1}
\end{equation}
by \eqref{eq:calculations_Doob_L1_X_0_geq_1}. This is \eqref{eq:doob_L1} in the case $X_0 < 1$.

Now we prove that Doob's $L^1$-inequality is attained.
This was also proven by \citet[Section 4]{peskir1998}. 
Firstly, let $X_0 \geq 1$.
Then the martingale 
\begin{align}
X = \left( B_{ \frac{t}{T-t} \wedge \tau_{\frac{1}{e}} }\right)_{t \leq T },\quad \textrm{where }\tau_{\frac{1}{e}} = \inf\{t: e B_t\leq \overline B_t\},
\end{align}
and $B$ is a Brownian motion with $B_0 = X_0$,
achieves equality in both \eqref{eq:Doob_improvement_2} and \eqref{eq:doob_L1}.
Secondly, let $X_0 < 1$.
Then the submartingale $X$ defined by
\begin{align}
\begin{cases} 
	X_0 &\mbox{if } t < T/2, \\
	B_{ \frac{t-T/2}{T/2-(t-T/2)} \wedge \tau_{\frac{1}{e}} } & \mbox{if } t \geq T/2, 
\end{cases}
\end{align}
where $B$ is a Brownian motion, $B_0 = 1$,
achieves equality in both, \eqref{eq:Doob_improvement_2} and \eqref{eq:doob_L1}.
 
Finally, we note that in the calculations \eqref{eq:calculations_Doob_L1_X_0_geq_1} which led to \eqref{eq:doob_Lp} there are three inequalities:
the first one comes from Theorem \ref{thm:main_result_1} and does not concern the claim regarding \eqref{eq:strict_error_Doob_L1_1}--\eqref{eq:strict_error_Doob_L1_3}.
The second one is clearly strict if and only if \eqref{eq:strict_error_Doob_L1_1} holds.
The third one is clearly strict if and only if \eqref{eq:strict_error_Doob_L1_2} holds.
In addition, in the case $X_0<1$ there is an additional error coming from \eqref{eq:calculations_Doob_L1_X_0_leq_1}.
Note that 
\begin{align*}
\restr{ \frac{ \E{ \left( X_T - \zeta \right)^+ } }{ y-\zeta } }{\zeta = \infty } := \lim_{\zeta \to -\infty} \frac{ \E{ \left( X_T - \zeta \right)^+ } }{ y-\zeta } = 1
\end{align*}
in the case when $\E{ \bar{X}_T } < \infty$.
Hence, the first inequality in \eqref{eq:calculations_Doob_L1_X_0_leq_1} is strict if and only if \eqref{eq:strict_error_Doob_L1_3} holds. 
The second inequality in \eqref{eq:calculations_Doob_L1_X_0_leq_1} is strict if and only if \eqref{eq:strict_error_Doob_L1_1} or \eqref{eq:strict_error_Doob_L1_2} holds.
\end{proof}

\subsection{Doob Type Inequalities, \texorpdfstring{$0 < p<1$}{0<p<1}}

It is well known that if $X$ is a positive continuous local martingale converging a.s. to zero, then
\begin{align}
\bar{X}_{\infty} \sim \frac{X_0}{U}
\end{align}
where $U$ is a uniform random variable on $[0,1]$.
Further, if $X$ does not converge to zero but to a non-negative limit $X_{\infty}$, we can, possibly on an enlarged probability space, extend it to a positive continuous local martingale $Y$ converging a.s. to zero and clearly $\bar{X}_{\infty} \leq \bar{Y}_{\infty}$.

Hence, for any positive continuous local martingale $X$,
\begin{align}
\E{ \bar{X}^p_{T} } \leq \E{ \left( \frac{X_0}{U} \right)^p } = \int_0^1 \left( \frac{X_0}{u} \right)^p \dd u = \frac{X_0^p}{1-p}
\label{eq:bound_doob_type_classical_theory}
\end{align}
and \eqref{eq:bound_doob_type_classical_theory} is attained.
We now generalize \eqref{eq:bound_doob_type_classical_theory} to a non-negative submartingale.

\begin{Proposition}[Doob Type Inequalities, $0 < p<1$]
\label{prop:Doobsmallp}
Let $X$ be a non-negative \cadlag submartingale, $X_0>0$, and $p\in (0,1)$. Denote $m_r:=X_0^{-r}\E{X_T^r}$ for $r\le 1$. Then: 
\begin{itemize}
	\item[{\rm (i)}]  there is a unique $\hat{\alpha}\in(0,1]$ which solves
	\begin{align}
	m_p\hat\alpha^{-p}=\frac{1-p+pm_1}{1-p+p\hat\alpha}
	\label{eq:alpha_hat}
\end{align}	 
and for which we have
\begin{subequations}
\begin{align}
\E{ \bar{X}_T^p } 
\;&\le\; 
X_0^p m_p \hat{\alpha}^{-p}  = \frac{X_0^p}{1-p+p\hat\alpha} + X_0^{p-1} \frac{p}{1-p+p\hat\alpha} \Big( \E{X_T} - X_0 \Big)
\label{eq:Doob_type_inequality_a} \\
\;&<\;
\frac{X_0^p}{1-p} + X_0^{p-1} \frac{p}{1-p} \Big( \E{X_T} - X_0 \Big).
\label{eq:Doob_type_inequality_b}
\end{align}
\end{subequations}
	\item[{\rm (ii)}] there exists a martingale which attains equality in  \eqref{eq:Doob_type_inequality_a}. Further, for every $\epsilon > 0$ there exists a martingale such that
	\begin{align}\label{eq:Doob_psmall_epsilon}
	0 \leq \frac{X_0^p}{1-p} + X_0^{p-1} \frac{p}{1-p} \Big( \E{X_T} - X_0 \Big) - \E{ \bar{X}_T^p } < \epsilon.
	\end{align}
\end{itemize}
\end{Proposition}

\begin{proof}
Following the calculations in \eqref{eq:estimations_Doob_Lp}, we see that
 \begin{eqnarray*}
 \E{\bar X_T^p}
 \leq
\frac{1}{1-\alpha} X_0^p
 +\frac{1}{(1-\alpha)(1-p)}
   \E{-\alpha^{1-p} X_T^p
        +p X_0^{p-1}X_T
       }
 =
 X_0^p f(\alpha),
 \end{eqnarray*}
where, with the notation $m_r$ introduced in the statement of the Proposition,
 \begin{eqnarray*}
 f(\alpha)
 :=
 \frac{1}{1-\alpha}
                  +\frac{-\alpha^{1-p}m_p+pm_1}
                           {(1-\alpha)(1-p)},
 \qquad \alpha\in[0,1].
 \end{eqnarray*}
Next we prove the existence of a unique $\hat\alpha\in(0,1]$ such that $f(\hat\alpha)=\min_{\alpha \in [0,1]}f(\alpha)$.
To do this, we first compute that 
 \begin{eqnarray*}
 f'(\alpha)
 =
 \frac{h(\alpha)}{(1-p)(1-\alpha)^2},
 &\mbox{where}&
 h(\alpha)
 :=
 1-p+pm_1-(1-p+p\alpha)m_p\alpha^{-p}.
 \end{eqnarray*}
By direct calculation, we see that $h$ is continuous and strictly increasing on $(0,1]$, with $h(0+)=-\infty$ and $h(1)=1-p+pm_1-m_p$. Moreover, it follows from the Jensen inequality and the submartingale property of $X$ that $m_p\le m_1^p$ and $m_1\ge 1$. This implies that $h(1)\ge 0$ since $1-p+px-x^p\geq 0$ for $x\geq 1$. In consequence, there exists $\hat\alpha\in(0,1]$ such that $h\le 0$ on $(0,\hat\alpha]$ and $h\ge 0$ on $[\hat\alpha,1]$. This implies that $f$ is decreasing on $[0,\hat\alpha]$ and increasing on $[\hat\alpha,1]$, proving that $\hat\alpha$ is the unique minimizer of $f$.

Now the first inequality \eqref{eq:Doob_type_inequality_a} follows by plugging the equation $h(\hat{\alpha}) = 0$ into the expression for $f$. The bound in \eqref{eq:Doob_type_inequality_b} is then obtained by adding strictly positive terms. It also corresponds to taking $\alpha=0$ in the expression for $f$. This completes the proof of the claim in {\rm (i)}. 

As for {\rm (ii)}, the claim regarding a martingale attaining equality in \eqref{eq:Doob_type_inequality_a} follows precisely as in the proof of Proposition \ref{prop:doob_Lp}.
Let $\alpha \in (0,1)$ and recall that $\tau_\alpha = \inf\{t: B_t\leq \alpha \bar B_t\}$ for a standard Brownian motion $B$ with $B_0=X_0>0$. Then, similarly to the proof of Proposition \ref{prop:doob_Lp}, we compute directly 
\begin{equation}\label{eq:tail_Xalpha}
 \PP(\bar{B}_{\tau_\alpha}\geq y)= \PP(B_{\tau_\alpha}\geq \alpha y) = \left(\frac{X_0}{y}\right)^{\frac{1}{1-\alpha}},\quad y\geq X_0.
\end{equation}
Computing and simplifying we obtain $\E{ \bar{B}^p_{\tau_{\alpha}} } = \frac{1}{1-p+p\alpha}X_0^p$, and hence $\E{ B^p_{\tau_{\alpha}} } = \frac{\alpha^p}{1-p+p\alpha}X_0^p$, while $\E{ B_{\tau_{\alpha}} } = X_0$. It follows that $\hat \alpha =\alpha$ solves \eqref{eq:alpha_hat} and equality holds in \eqref{eq:Doob_type_inequality_a}. Taking $\alpha$ arbitrarily small shows \eqref{eq:Doob_psmall_epsilon} holds true.
\end{proof}

We close this section with a new type of Doob's $L\ln L$ type of $L^1$ inequality obtained taking $p\nearrow 1$ in Proposition \ref{prop:Doobsmallp}.
Since $\hat\alpha(p)$ defined in \eqref{eq:alpha_hat} belongs to $[0,1]$ there is a converging subsequence. So without loss of generality, we may assume $\hat\alpha(p)\longrightarrow\hat\alpha(1)$ for some $\hat\alpha(1)\in[0,1]$. 
In order to compute $\hat\alpha(1)$, we re-write  \eqref{eq:alpha_hat} into  
 \begin{eqnarray}\label{h(p)-h(1)}
 \frac{g(p)-g(1)}{p-1}
 = m_p
 &\mbox{where}&
 g(p):=pm_p\hat\alpha(p)-(1-p+pm_1)\hat\alpha(p)^p.
 \end{eqnarray}
We see by a direct differentiation, invoking implicit functions theorem, that
 \begin{eqnarray*}
 g'(1)
 &=&
 \hat\alpha(1)\left(1+\E{\frac{X_T}{X_0}\ln \frac{X_T}{X_0}}\right)-\hat\alpha(1)\ln\hat\alpha(1)\E{\frac{X_T}{X_0}}.
 \end{eqnarray*}
Then, sending $p\to 1$ in \eqref{h(p)-h(1)}, we get the following equation for $\hat\alpha(1)$:
 \begin{eqnarray}\label{eq:alpha1}
 \hat\alpha(1)\left(1+\E{\frac{X_T}{X_0}\ln \frac{X_T}{X_0}}\right) 
 &=&
 \E{\frac{X_T}{X_0}} (1+ \hat\alpha(1)\ln\hat\alpha(1)).
 \end{eqnarray}
 We note that this equation does not solve explicitly for $\hat \alpha(1)$.  
Sending $p\to 1$ in the inequality of Proposition 3.4 we obtain the following improvement to the classical Doob's $L\log L$ inequality presented in Proposition \ref{prop:doob_L1} above.
\begin{Corollary}[Improved Doob's $L^1$ Inequality]
\label{cor:NewL1}
Let $X$ be a non-negative \cadlag submartingale, $X_0>0$. Then: 
 \begin{eqnarray}\label{eq:NewL1}
 \E{\bar{X}_T}
 &\le &
 \frac{\mathbb{E}[X_T]}{\hat\alpha} = \frac{\E{X_T\ln X_T}+ X_0 - \E{X_T}\ln X_0}{1+ \hat\alpha\ln\hat\alpha}
 \end{eqnarray} 
 where $\hat\alpha\in (0,1)$ is uniquely defined by \eqref{eq:NewL1}.
\end{Corollary}
Note that the equality in \eqref{eq:NewL1} is a rewriting of \eqref{eq:alpha1}.
To the best of our knowledge the above inequality in \eqref{eq:NewL1} is new. It bounds $\E{\bar{X}_T}$ in terms of a function of $\E{X_T}$ and $\E{X_T\ln X_T}$, similarly to the classical inequality in \eqref{eq:doob_L1}. However here the function depends on $\hat \alpha$ which is only given implicitly and not explicitly. In exchange, the bound refines and improves the classical inequality in \eqref{eq:doob_L1}. This follows from the fact that 
$$1+\alpha \ln \alpha \geq \frac{e-1}{e},\quad \alpha \in (0,1).$$
We note also that for $X_t:= B_{\frac{t}{T-t}\land \tau_\alpha}$, $\alpha\in (0,1)$, we have $\hat\alpha=\alpha$ and equality is attained in \eqref{eq:NewL1}. This follows from the proof above or is verified directly using \eqref{eq:tail_Xalpha}.
The corresponding classical upper bound in \eqref{eq:doob_L1} is strictly greater expect for $\alpha=1/e$ when the two bounds coincide.

\subsection{No Further Improvements with Intermediate Moments}

Next, we prove that beyond the improvement stated in Proposition \ref{prop:doob_Lp} no sharper bounds can be obtained from the inequalities of Theorem \ref{thm:main_result_1}.
\begin{Proposition}[No Improvement of Doob's $L^p$-Inequality from Theorem \ref{thm:main_result_1}]
\label{prop:No Improvement of Doob}
Let $p>1$ and $\tbzeta \in \tilde{\mathscr{Z}}$ be such that $\tilde{\zeta}_j(m) \neq \zeta_{\frac{p-1}{p}}(m) = \frac{p-1}{p}m$ for some $m>X_0$ and some $j$. 
Then, there exists a martingale $X$ such that
\begin{align}
\left( \frac{p}{p-1} \right)^p \E{ X_T^p } - \frac{p}{p-1}X_0^p 
< 
\mathrm{UB}\left( X, \mathrm{pow}^p, \tbzeta \right).
\end{align}
\end{Proposition}

\begin{proof}
Let $\alpha > \frac{p-1}{p} =: \alpha^{\star}$ and take $X^{\alpha}$ satisfying
\begin{align*}
0=X^{\alpha}_{t_1} = \dots = X^{\alpha}_{t_{j-1}}, \qquad \qquad  B_{\tau_{\alpha}} = X^{\alpha}_{t_j} = \dots = X^{\alpha}_{t_n}
\end{align*}
where $B$ is a Brownian motion started at $X_0$ and $\tau_{\alpha} = \inf\{ u>0 : B_u \leq \zeta_{\alpha}( \bar{B}_u ) \}$. 
It follows easily that for this process $X^{\alpha}$,
\begin{align*}
\mathrm{UB} \left( X^{\alpha}, \mathrm{pow}^p, \tilde{\zeta}_j \right)  
\leq \mathrm{UB} \left( X^{\alpha}, \mathrm{pow}^p, \tbzeta \right) 
\end{align*}
and hence it is enough to prove the claim for $n=1$ and $\tbzeta = \tilde{\zeta}_j$.

For all $\alpha \in (\alpha^{\star},\alpha^{\star}+\epsilon), \  \epsilon>0$, Proposition \ref{prop:positive_hedging_error} yields existence of a non-empty, open interval $\mathcal{I}_{\alpha}$ such that
\begin{align}
\forall m \in \mathcal{I}_{\alpha}: \quad 
\mathrm{UB} \left( X^{\alpha},\indic{[m,\infty)},\zeta_{\alpha} \right) < \mathrm{UB} \left( X^{\alpha},\indic{[m,\infty)}, \tilde{\zeta}_j \right).
\end{align}
In fact, taking $\epsilon>0$ small enough, $\mathcal{I}_{\alpha}$ can be chosen such that
\begin{align}
\bigcap_{\alpha \in (\alpha^{\star}, \alpha^{\star}+ \epsilon)} \mathcal{I}_{\alpha} \quad \supseteq \quad (m_1,m_2), \qquad \qquad X_0 < m_1 < m_2.
\end{align}
We can further (recalling the arguments in Case A and Case B in the proof of Proposition \ref{prop:positive_hedging_error}) assume that for all $\alpha \in (\alpha^{\star}, \alpha^{\star} + \epsilon)$:
\begin{align}
\forall m \in (m_1,m_2): \quad \mathrm{UB}\left( X^{\alpha},\indic{[m,\infty)},\tilde{\zeta}_j \right) - \mathrm{UB}\left( X^{\alpha},\indic{[m,\infty)}, \zeta_{\alpha} \right) \geq \delta > 0.
\end{align}
The claim follows by letting $\alpha \downarrow \alpha^{\star}$ and using the asymptotic optimality of $\left( X^{\alpha} \right)_{\alpha}$, see \eqref{eq:sharpness_doob_Lp}.
\end{proof}

In addition to the result of Proposition \ref{prop:No Improvement of Doob} we prove that there is no \enquote{intermediate moment refinement of Doob's $L^p$-inequalities} in the sense formalized in the next Proposition.
Intuitively, this could be explained by the fact that the $p^{\mathrm{th}}$ moment of a continuous martingale is continuously non-decreasing and hence does not add relevant information about the $p^{\mathrm{th}}$ moment of the maximum. 
Only the final $p^{\mathrm{th}}$ moment matters in this context. 
\begin{Proposition}[No Intermediate Moment Refinement of Doob's $L^p$-Inequality]
\label{prop:no_n_moment_refinement}
If $a_1,\dots,a_n$ are such that for every continuous submartingale $X$, $X_0 = 0$, we have
\begin{align}
\E{ \bar{X}_T^p }  \leq  \sum_{i=1}^{n} a_i \E{ |X_{t_i}|^{p} }
\label{eq:n_moment_doob}  
\end{align}
or
\begin{align}
\E{ \bar{X}_T^p }  \leq  \sum_{i=1}^{n} a_i \E{ |X_{t_i} - X_{t_{i-1}}|^{p} },
\label{eq:n_moment_doob_diff}  
\end{align}
then
\begin{align}
\left( \frac{p}{p-1} \right)^p   \E{ |X_{T}|^{p} }   \leq  \sum_{i=1}^{n} a_i \E{ |X_{t_i}|^{p} }
\label{eq:Doob no improvement} 
\end{align}
or
\begin{align}
\left( \frac{p}{p-1} \right)^p   \E{ |X_{T}|^{p} }   \leq  \sum_{i=1}^{n} a_i \E{ |X_{t_i} - X_{t_{i-1}}|^{p} }, 
\label{eq:Doob no improvement diff} 
\end{align}
respectively.
\end{Proposition}

\begin{proof}
From \citet[Example 4.1]{peskir1998} or our Proposition \ref{prop:doob_Lp} we know that Doob's $L^p$-inequality given in \eqref{eq:doob_Lp} is enforced by a sequence of continuous martingales $(Y^{\epsilon})$ in the sense of \eqref{eq:sharpness_doob_Lp}.
We will write $\left( \frac{p}{p-1} \right)^p   \E{ |Y^{\epsilon}_{T}|^{p} } \simeq \E{ \max_{t\leq T}|Y^{\epsilon}_t|^p }$.

Firstly consider the case of \eqref{eq:n_moment_doob} and \eqref{eq:Doob no improvement}.

By scalability of the asymptotically optimal martingales $(Y^{\epsilon})$ we can assume
\begin{align*}
\E{ |X_{t_n}|^{p} } = \E{ |Y_{t_n}^{\epsilon}|^p }.
\end{align*}
In addition we can find times $u_1 \leq \dots \leq u_{n-1}$ such that
\begin{align*}
\E{ |X_{t_i}|^{p} } = \E{ |Y_{u_i}^{\epsilon}|^p }.
\end{align*}
Therefore, writing $u_n = t_n = T$ and using asymptotic optimality of $(Y^{\epsilon})$, 
\begin{alignat*}{3}
\left( \frac{p}{p-1} \right)^p   \E{ |X_{t_n}|^{p} } &= &&\left( \frac{p}{p-1} \right)^p   \E{ |Y^{\epsilon}_{u_n}|^{p} } \simeq \E{ \max_{t\leq T}|Y^{\epsilon}_t|^p }  \\
&\stackrel{\mathclap{\eqref{eq:n_moment_doob}}}{\leq} \quad &&\sum_{i=1}^{n} a_i \E{ |Y^{\epsilon}_{u_i}|^{p} } = \sum_{i=1}^{n} a_i \E{ |X_{t_i}|^{p} }.
\end{alignat*}
Equation \eqref{eq:Doob no improvement} follows.

Secondly consider the case of \eqref{eq:n_moment_doob_diff} and \eqref{eq:Doob no improvement diff}.
Taking a martingale which is constant until time $t_{i-1}$ and after $t_{i}$ and using the the fact that Doob's $L^p$ inequality is sharp yields 
\begin{align*}
\left(\frac{p}{p-1} \right)^p \leq a_i \qquad \text{for all $i=1,\dots,n$.} 
\end{align*}
Equation \eqref{eq:Doob no improvement diff} follows.
\end{proof}

\begin{Remark}
Analogous statements hold for Doob's $L^1$ inequality. 
This can be argued in the same way by using that Doob's $L^1$ inequality is attained (cf. e.g. \citet[Example 4.2]{peskir1998} or our Proposition \ref{prop:doob_L1}) and observing that the function $x \mapsto x \log(x)$ is convex.
\end{Remark}

\bibliographystyle{hapalike}
\bibliography{literature_papers3}

\end{document}